\documentclass[english, reqno]{amsart}
\usepackage{mathpazo}
\usepackage[T1]{fontenc}
\usepackage[latin9]{inputenc}
\usepackage{color}
\usepackage{babel}
\usepackage{perpage}
\usepackage{amsthm}
\usepackage{amssymb}
\usepackage{esint}
\usepackage{nicefrac}
\usepackage[unicode=true, pdfusetitle,
 bookmarks=true,bookmarksnumbered=false,bookmarksopen=false,
 breaklinks=false,pdfborder={0 0 0},backref=false,colorlinks=true,linkcolor=black,citecolor=black,urlcolor=blue]
 {hyperref}

\MakePerPage{footnote} 

\makeatletter
  \theoremstyle{plain}
  \newtheorem{thm}{Theorem}
  \newtheorem*{clm*}{Claim}
  \newtheorem*{lem*}{Lemma}
  \newtheorem{lem}{Lemma}

\usepackage{calrsfs}
\newcommand{\ZZ}{\mathbb{Z}}
\newcommand{\PP}{\mathbb{P}}
\newcommand{\Cl}{\mathcal{C}}
\newcommand{\Nn}{\mathcal{N}}
\newcommand{\parv}{\partial_{\textrm v}}
\makeatother

\begin{document}

\title{Uniqueness of percolation on products with \texorpdfstring{$\mathbb Z$}{Z}}

\author{Itai Benjamini and Gady Kozma}

\begin{abstract}
We show that there exists a connected graph $G$ with subexponential volume growth such that
critical percolation on $G\times\mathbb{Z}$ has infinitely many infinite
clusters. We also give some conditions under which this cannot occur.
\end{abstract}

\maketitle

This paper begins with the observation that if $G$ is any connected
graph and $p$ is any number in $[0,1]$, then the number of infinite
clusters in $p$-percolation\footnote{For background on percolation see
  \cite{G99} or \cite{LP}.} on $G\times \mathbb{Z}$ is deterministic,
and is either 0, 1 or $\infty$. The proof is an easy consequence of
the fact that one can take any finite set of vertices and translate it
along the $\mathbb{Z}$ axis and get a set of variables disjoint from
the one you started with. 


In view of this, Sznitman asked\footnote{Personal communication.} whether the Burton-Keane argument
\cite{BK89} applies. Namely, assume $G$ is amenable, does it follow
that $G\times \mathbb{Z}$ has only finitely many infinite clusters?
The definition of amenability used here is that the Cheeger constant
is 0, namely, for every $\epsilon>0$ there is some finite set of
vertices $A$ such that $|\partial A|\le \epsilon |A|$ where $\partial
A$ is the edge boundary of $A$.

As stated the answer is no. A binary tree with a infinite
path added at the root serves as a counterexample.  We suggest a
slight modification.

\medskip
Say that $G$ is {\it strongly amenable} if $G$ contains no nonamenable subgraph.
\medskip

Assume $G$ is strongly amenable, can
one find an interval $[p_1,p_2]$ such that percolation on $G\times
\mathbb{Z}$ has infinitely-many infinite clusters for every $p$ in
this interval? What if we further assume that $G$ has polynomial
volume growth? 

Our main result is to construct an example of a strongly amenable
graph of the form $G\times\mathbb{Z}$ with non uniqueness at $p_c$. We
do not see yet any example of such a graph in which no percolation
occurs at $p_c$, but non uniqueness occurs for some $p>p_c$.

It is tempting to reformulate this question as $p_c=p_u$ (where
$p_u$ is the threshold for uniqueness, see
\cite{LP} for the precise definition) but there is no monotonicity of
uniqueness for graphs of the type $G\times \mathbb{Z}$. Indeed,
connect the root of a $\mathbb{Z}^{99}$ lattice to the root of a $10$
regular tree $T$, denote this graph by $G$. The parameters were chosen
so as to satisfy
\[
p_c(\mathbb{Z}^{100})<p_c(T\times\mathbb{Z})<p_u(T\times\mathbb{Z})
\]
(the first inequality follows from the fact that $p_c(\ZZ^d)\le C/d$,
see \cite{K90} or \cite[\S 4]{ABS04}; and from the bound
$p_c(T\times\mathbb{Z}) \ge \frac{1}{11}$ which holds for any graph
with degree $12$. The second inequality follows from \cite{S01}).
It is not hard to see that on $ G \times \mathbb{Z}$ for small $p$ no
percolation occurs. Then between $p_c(\mathbb{Z}^{100})$ and
$p_c(T\times \mathbb{Z})$ there is a unique infinite cluster. Between
$p_c(T\times\mathbb{Z})$ and $ p_u(T \times \mathbb{Z})$ there are infinitely many infinite clusters. Finally, above $ p_u(T \times \mathbb{Z})$ again one
has a unique infinite cluster. This example can be generalized to an
arbitrary, even infinite number of transitions.

Regarding strong amenability we ask the following.
There is no known example of an exponentially growing Cayley graph which is strongly amenable
and the existence of such is still open, (see \cite{CT08} for recent related work and a review of what is known
for groups).
A graph $G$ has uniform growth if all balls
with the same radius have the same size up to a fixed multiplicative constant.
Is there a graph with a uniform exponential growth which is strongly amenable?

This notes has two results on this problem. The first is a counterexample:
\begin{thm}\label{thm:cntr}
There exists a connected graph $G$ with subexponential volume growth such that
critical percolation on $G\times\mathbb{Z}$ has infinitely many infinite
clusters.\end{thm}

The second is a positive result, a family of graphs $G$ for which we
can prove that $G\times\ZZ$ does not have infinitely-many infinite
clusters at any $p$. The result is not very satisfying, and calls for strengthening.
\begin{thm}\label{thm:cutsets}Let $G$ be a connected graph such that each finite set can be
  disconnected from infinity by removing a bounded number of
  edges. Then $G\times\ZZ$ does not have infinitely-many infinite
  components.
\end{thm}
In other words, we require from $G$ that there exists some constant
$K$ such that for every finite set of vertices $A$ one can find $K$ edges
$e_1(A),\dotsc,e_K(A)$ such that removing these edges will make all the
components of all $v\in A$ finite.

\section*{Proof of theorem \ref{thm:cntr}}
\noindent Let $d$ be some sufficiently large number to be fixed
later. The graph is constructed as follows. Take a tree of degree
$4d$. Let $l_{1}=1$ and $l_{n+1}=l_{n}+\left\lceil d^{2}\log (n+1)\right\rceil $.
Now, for each $n\ge n_{0}$ ($n_{0}$ to be fixed later too, depending
on $d$) and for
each edge $(x,y)$ where $x$ is in level $l_{n}-1$ and $y$ is in
level $l_{n}$, disconnect $(x,y)$ and instead take a copy of $\mathbb{Z}^{d}$
(considered as a graph with the usual structure) and
connect $x$
with the vertex $(0,\dotsc,0)$ and $y$ with the vertex $(n,\dotsc,n)$.
All copies of $\mathbb{Z}^d$ (for all such $(x,y)$) are disjoint.
This terminates the definition of the graph $G$.

We will show that at $p=p_{c}(\mathbb{Z}^{d+1})$ the graph $G\times\mathbb{Z}$
has infinitely many infinite clusters. One can rather easily convince
oneself that in fact below $p$ our graph $G\times\mathbb{Z}$ has no infinite
clusters, so $p=p_{c}(G\times\mathbb{Z})$, but we will not do it
here. Note that $p=(1+o(1))/2d$ where $o(1)$ is as $d\to\infty$ \cite{K90}.

\subsection*{Subexponential growth}
Examine the ball $B$ of radius $r$ around the root of the tree we started
with. Now, for any $x$ in level $l_n$ of the tree, $d(x,0)\approx n^2$ because the
shortest path wastes $k$ steps between levels $l_{k}-1$ and $l_k$ for each $k<n$. Therefore $B$
contains tree elements up to level $l_h$ for $h\approx
\sqrt{r}$. Since $l_{h}\approx \sqrt{r}\log r$ we get that $B$ contains $\le
\exp(C\sqrt{r}\log r)$ tree vertices. The non-tree vertices of $B$ are
contained in $\le \exp(C\sqrt{r}\log r)$ copies of a $d$-dimensional
ball of radius $r$, so all in all we get
\[
|B|\le Cr^de^{C\sqrt{r}\log r}\le Ce^{C\sqrt{r}\log r}
\]
which is subexponential, as needed.

\subsection*{Existence}
We now turn to show that there are infinite clusters. Let $\gamma$ be
some path in $G\times\mathbb{Z}$. We say that $\gamma$ is ``between
levels $l_{n-1}$ and $l_n$'' if for each vertex $(v,n)$ of $\gamma$,
either $v$ is in the tree, and its level is between $l_{n-1}$ and
$l_n$, or $v$ belongs to one of the copies of $\mathbb{Z}^d$ that were
connected between levels $l_n-1$ and $l_n$. Further we require that
only the first and last vertices of $\gamma$ may have their $v$ in levels
$l_{n-1}$ and $l_n$. The interior vertices need to be in levels
strictly between, or in the copies of $\mathbb{Z}^d$.
With this
definition we have
\begin{lem}\label{lem:CdomU}Let $d$ be sufficiently large, $n>n_0$ and let $x$ be a tree element in level
  $l_{n-1}$. Let $Z$ be the set of 
  vertices $z$ in level $l_n$ such that $(z,0)$ is connected to
  $(x,0)$ by an open path between $l_{n-1}$ and $l_n$. Then $|Z|$
  stochastically dominates a variable $U$, independent of $n$, with $\mathbb{E} U>1$.
\end{lem}
\begin{proof}
Examine the set $Y$ of vertices $y$ of $G$
in level $l_{n}-1$ such that $(x,0)\in G\times \mathbb{Z}$ is
connected to $(y,0)$ inside the ``slice'' $G\times\{0\}$. This is just a
problem on supercritical branching processes (for $d$ sufficiently
large $p_c$ of the tree, which is $1/(4d-1)$, is smaller than
$p_c(\mathbb{Z}^{d+1})=(1+o(1))/2d$) and a standard second
moment argument gives that
\[
\mathbb{P}\left(|Y|>\tfrac 12 ((4d-1)p)^{d^2\log n}\right)>c
\]
where the term $d^2\log n$ is simply $l_{n}-1-l_{n-1}$, the height of the
tree we are examining. Here and below $c$ denotes positive constants which are
allowed to depend only on $d$. For $d$ sufficiently large we may
replace the term $(4d-1)p$ with $3/2$ and drop the the $\frac 12$
before it. We get
\[
\PP(|Y|>n^{cd^2})>c.
\]

Examine next the set $Z(y)$ of vertices $z$ in level $l_{n}$ such that
$(y,0)$ is connected to $(z,0)$ by an open path that starts by moving
from $(y,0)$ into an element $(0,\dotsc,0)$ in one of the copies of
$\ZZ^d\times\ZZ$ ``below'' it, then winds around in that copy and finally takes the last step from
$(n,\dotsc,n,0)$ to $(z,0)$ (this time we allow the path to use the extra dimension, i.e.\ it is not restricted to the slice $G\times\{0\}$). By \cite{H08}, the probability that $(0,\dotsc,0)\leftrightarrow (n,\dotsc,n,0)$ in $\mathbb{Z}^d\times \mathbb{Z}=\mathbb{Z}^{d+1}$ is $\ge c
n^{2-d}$, recall that we are examining
$p_c(\mathbb{Z}^{d+1})$. Further, all these events (for different $y$)
are independent, because they examine disjoint copies of
$\mathbb{Z}^{d+1}$. With the
argument of the previous paragraph we see that $(x,0)$ has probability
$>c$ to have $n^{cd^2}$ ``children'' in level $l_{n}-1$ and each one
has probability $>cn^{2-d}$ to have a child in level $l_{n}$, independently. In other
words, $Z=\cup_{y\in Y}Z(y)$ 
dominates a random variable which is with 
probability $1-c$ empty, and with probability $c$ a sum of
$n^{cd^2}$ independent Bernoulli variables with probability
$cn^{2-d}$. Hence, for $n$ sufficiently large, $Z$ dominates a
variable $U$ which is empty with probability $1-c/2$ and with
probability $c/2$ is $4/c$. The lemma is thus proved.
\end{proof}
The existence of an infinite cluster at $p$ now follows. Examine a
vertex $(x,0)$ with $x$ in level $n$ for $n>n_0$. Define inductively sets of
vertices $X_i$ with $X_0=\{(x,0)\}$ and $X_i$ being all vertices
$(y,0)$ with $y$ in level $l_{n+i}$ which are connected to some
vertex in $X_{i-1}$ by an open path between levels $l_{n+i-1}$ and
$l_{n+i}$. By lemma \ref{lem:CdomU}, the number of elements in $X_i$ which are
connected to a given element in $X_{i-1}$, stochastically dominates
the variable $U$. Further, all these connection events are independent
i.e.\ if $x$ and $x'$ are different elements in $\bigcup X_i$ then the
set of their decsendants are independent events, because the
connections use different edges. Hence the process $X_i$ dominates an
independent branching process with offspring distribution $U$. By
lemma \ref{lem:CdomU},
$\mathbb{E} U>1$ so by standard results, a branching process with
distribution $U$ survives with positive probability. Hence the process
$X_i$ also survives with positive probability. But if $X_i$ survives
to infinity then the cluster of $(x,0)$ is infinite. So the
probability that an infinite cluster exists is positive. As remarked
above, this is a 0-1 event, so in fact the probability is 1.

Below we will also
need that the probability has a uniform lower bound, so let us note it now: there
exists some constant $c>0$ such that
\begin{equation}\label{eq:unif}
\mathbb{P}((x,k)\mbox{ is in an infinite cluster}) > c\qquad \forall
x\in l_n\;\forall n>n_0\;\forall k.
\end{equation}

\subsection*{Non-uniqueness}To see that there are infinitely many clusters we apply the approach
of Benjamini and Schramm \cite[Theorem 4]{BS96} by comparing to a
branching random walk. We will not use usual branching random walk but a
slightly different process. Let us describe it.
\smallskip

1. If we have a particle in some vertex $(x,k)$ for $x$ in the tree, it
sends one particle to each neighbour of $(x,k)$ in $G\times\mathbb{Z}$ with probability
$p$. In particular, if $x$ is in level $l_n-1$ then a
particle is sent to each copy of $\mathbb{Z}^{d+1}$ ``below'' it, and if it is
in level $l_n$ then one particle is sent to the copy of
$\mathbb{Z}^{d+1}$ ``above'' it. 
\smallskip

2. Now assume we have a particle in
$(0,\dotsc,0,k)$ in some copy of
$\mathbb{Z}^{d+1}$ in the $n$\textsuperscript{th} level. It sends two
kinds of particles. First, one
particle with probability $p$ to its tree neighbour $(x,k)$ (which is
``above'' it in level $l_n-1$). Second,
it sends particles to all vertices $(y,l)$ with $y$ being in the same copy
of $\mathbb{Z}^{d}$ and equal to either $(0,\dotsc,0)$ or
$(n,\dotsc,n)$, with the distribution of descendants identical to that
of vertices connected to $(0,\dotsc,0,k)$ by independent percolation
at $p$ in $\mathbb{Z}^{d+1}$. 
\smallskip

3. A particle in $(n,\dotsc,n,k)$ does the
same, sending one particle with probability $p$ to its tree neighbour
$(x,k)$ ``below'' it, and extra particles to $(0,\dotsc,0,l)$ and
$(n,\dotsc,n,l)$ with the percolation distribution.
\smallskip

This ends the description of the process. Denote the step of particles
at time $t$ by $X_t$. The proof of
non-uniqueness now follows from the following two claims
\begin{lem}\label{lem:XdomC} For any $x\in G\times\mathbb{Z}$, the set of vertices
  visited by $X_t$ at some $t$ stochastically dominates
  $\mathcal{C}(x)\setminus Z$ where $Z$ is the union of all copies of
  $\mathbb{Z}^{d+1}$.
\end{lem}

(here and below $\mathcal{C}(x)$ denotes the cluster of $x$).

\begin{lem}\label{lem:transient} For $d$ and $n_0$ sufficiently
  large ($n_0$ depending on $d$), $X_t$ is transient i.e.\ the expected
  number of returns to the starting point is finite.
\end{lem}

\begin{proof}[Proof of non-uniqueness given lemmas \ref{lem:XdomC} and
    \ref{lem:transient}]
Assume by contradiction that
there is only one infinite cluster. That would imply, for any $x$ and
$y$ in $G$.
\begin{multline*}
\mathbb{P}((x,0)\leftrightarrow (y,0))\ge 
\mathbb{P}(\{|\mathcal{C}((x,0))|=\infty\}\cap\{|\mathcal{C}((y,0))|=\infty\})\ge\\
\ge\mathbb{P}(|\mathcal{C}((x,0))|=\infty)\mathbb{P}(|\mathcal{C}((y,0))|=\infty)
\end{multline*}
where the second
inequality follows from FKG (see \cite[\S 2.2]{G99}). Assuming $x$
and $y$ belong to level $l_n$  (not necessarily the same $n$ for $x$
and $y$), (\ref{eq:unif}) would give
\[
\mathbb{P}((x,0)\leftrightarrow (y,0))\ge c.
\]
On the other hand, our process $X$ is transient (lemma
\ref{lem:transient}) and symmetric, i.e.\ the probability to reach $x$
from $y$ is the same as reaching $y$ from $x$. Any such process must
satisfy that, when we fix the starting point of $X$,
\begin{equation}
\lim_{y\to\infty}\PP\left(y\in\bigcup X_t \right)=0
\end{equation}
since otherwise you will have a sequence $y_n\to\infty$ such that you
can return to your the starting point with probability $>c$ after
visiting $y_n$. This clearly contradicts transience. 

Now apply the domination result. If we also assume $y$ is
not in the copies of $\ZZ^{d+1}$ then we get
\begin{equation}\label{eq:transience}
\lim_{y\to\infty}\mathbb{P}((x,0)\leftrightarrow (y,0))=0.
\end{equation}
We have reached a contradiction, demonstrating that one cannot
have a unique infinite cluster, and thus proving the theorem.
\end{proof}

\begin{proof}[Proof of lemma \ref{lem:XdomC}] This is completely
  standard: one simply explores the cluster using breadth-first search
  and note that the ``past'' of the algorithm only blocks you from
  exploring some vertices, while the branching process has no such
  restriction. One has to adapt the breadth-first search to our
  branching process i.e.\ when it enters a copy of $\ZZ^{d+1}$, search
  all neighbours in the two lines which connect outside in one step,
  but other than that there is no change necessary in the standard
  proof (see e.g.\ \cite{BS96}).
\end{proof}

\begin{proof}[Proof of lemma \ref{lem:transient}]We will show that
even the projection of $X_t$ on $G$ is transient. Since $X_t$ avoids
the copies of $\ZZ^d$ (except for the points directly connected to
the tree), let us consider the graph $H$ which is the tree of
degree $4d$, with every edge between level $l_n-1$ and level $l_n$
``stretched'' i.e.~replaced by a line with three edges and two
vertices. The projection of $X_t$ to $G$ is equivalent to a process on
$H$ that, from every particle, sends particles to all
neighbours with probability $p$, and sometimes sends additional particles
to itself and to one of its neighbours (above or below, depending on whether
you are in one of the stretched levels, and where exactly you are in
them). By lemma \ref{lem:Remco} below, the
expected number of the additional particles that remain in place is
$<\frac{1}{2}$, if only $d$ is chosen sufficiently large. As for
the additional particles sent to the neighbouring vertices, these can
be bounded directly from the two-point function \cite{HHS03}
i.e.\ from $\PP(0\leftrightarrow x)\le C(d)|x|^{2-d}$, and
if only $n_{0}$ is sufficiently large (as a function of $d$), the
expectation of these can be bounded by $1/4d$.

As a final simplification, embed
$H$ in a $4d$-regular tree by ``filling'' the two sparse rows between
$l_n-1$ and $l_n$. Namely, level $l_n-1$ of $H$ goes to a subset of level
$l_n-1+2(n-n_0)$ of the tree, the vertices on the stretched edges go
to levels $l_n+2(n-n_0)$ and $l_n+1+2(n-n_0)$ of the tree, level $l_n$
of $H$ goes to level $l_n+2+2(n-n_0)$ of the tree etc. The process on 
$H$ is now stochastically dominated by a process on the tree, which
sends from \emph{each} vertex of the tree to each of its neighbours a
particle with probability $p$ (like the ``usual'' vertices of $H$),
and also sends to itself and to \emph{all} of its neighbors additional particles
(like the vertices of $H$ on the stretched edges). In short, each
vertex sends both the particles it would have sent if it were a vertex
of the tree and the particles it would have sent if it were a vertex
of the stretched edges.

We have reached
now a very well-understood process: a branching random walk on a $4d$-regular
tree, where each particle sends to each of its neighbours an expected
$p+1/4d<1/d$ offspring, and an expected $<\frac{1}{2}$ offspring remain
in place. Showing that this process is transient can be done with
a straightforward calculation. Fix some $t$ and examine the number
of particles still at the origin at time $t$. Any such particles
must have done $s$ steps on the tree (for some $s\le t$) and stayed
in place $t-s$ steps. For a fixed $t$ and $s$ the expected number
of offspring is less than or equal to\[
\#\{\mbox{paths in the tree of length \ensuremath{s} returning to $x$}\}\cdot\frac{1}{d^{s}}\cdot\frac{1}{2^{t-s}}=(4d)^{s/2}\cdot\frac{1}{d^{s}}\cdot\frac{1}{2^{t-s}}\]
and summing over all $t$ and $s$ shows that the process is transient,
if only $d$ is sufficiently large (this last step requires $d>4$, but
we also rely on $p_c(\ZZ^d)$ being sufficiently close to $1/2d$ and on
lemma \ref{lem:Remco}, both which requires larger $d$).
\end{proof}
\begin{lem}\label{lem:Remco}
Let $d$ be sufficiently large. Then critical percolation on $\mathbb{Z}^{d}$
satisfies\[
\sum_{n\ne0}\mathbb{P}(\vec{0}\leftrightarrow(0,\dotsc,0,n))<\frac{C}{\sqrt{d}}\]
(here $\vec{0}=(0,\dotsc,0)\in\mathbb{Z}^{d}$ and $C$ is a constant
independent of the dimension)\end{lem}
We assume that the correct asymptotic behaviour is $C/d$, but we do
not need it in this paper. It is well known that in $d=2$ this sum is
$\infty$, for example, it follows from the estimate
$\mathbb{P}(0\leftrightarrow\partial B(n))\ge cn^{-1/3}$, see
\cite[equation (5.1)]{K87}. We will not give more details on this
fact, as it will take us too far off course.
\begin{proof}
We follow Heydenreich, van der Hofstad and Sakai \cite{HHS}. Let
us recall formula (1.41) ibid., in their notation:
\begin{equation}
\widehat{G_{z_{c}}}(k)=\frac{1+O(\beta)}{1-\widehat{D}(k)}.\label{eq:HHS}
\end{equation}
Let us explain the notation. The $z_{c}$ is $2d\cdot p_{c}$ \cite[\S 1.2.3]{HHS} and $G$
is the connection probability \[
G_{z}(x)=\mathbb{P}(0\leftrightarrow x)\]
where the probability is with respect to percolation at $z/2d$ \cite[\S 1.2.3 and equation (1.19)]{HHS}
so $G_{z_{c}}$ is the critical connection probability. $\beta=K/d$
where $K$ is some absolute constant \cite[first line of \S 1.3]{HHS}
and the constant implicit in the $O(\cdot)$ is also dimension independent.
$D(x)=\frac{1}{2d}\mathbf{1}_{\{|x|=1\}}$ \cite[(1.1)]{HHS}. Finally
$\,\widehat{\cdot}\,$ is the usual Fourier transform and $k\in[-\pi,\pi)^{d}$.
In particular\[
\widehat{D}(k)=\frac{1}{d}\sum_{i=1}^{d}\cos(k_{i}).\]

With (\ref{eq:HHS}) explained, let us calculate first the sum including
the term $\mathbb{P}(\vec{0}\leftrightarrow\vec{0})$,\begin{align*}
\sum_{n}\mathbb{P}(\vec{0}\leftrightarrow(0,\dotsc,0,n)) & =\frac{1}{(2\pi)^{d-1}}\int\widehat{G_{z_{c}}}(k_{1},\dotsc,k_{d-1},0)\, dk_{1}\dotsb dk_{d-1}\\
 & \le\frac{1}{(2\pi)^{d-1}}\int\frac{1+C/d}{\frac{d-1}{d}-\frac{1}{d}\sum_{i=1}^{d-1}\cos(k_{i})}\, dk_{1}\dotsb dk_{d-1}\\
 & =\frac{d+C}{d-1}\cdot\frac{1}{(2\pi)^{d-1}}\int\frac{dk}{1-\widehat{D_{d-1}}(k)}.\end{align*}
Removing the term $\mathbb{P}(\vec{0}\leftrightarrow\vec{0})=1=\frac{1}{(2\pi)^{d-1}}\int1$
gives\begin{align}
\sum_{n\ne0} & \le\frac{d+C}{d-1}\cdot\frac{1}{(2\pi)^{d-1}}\int\frac{1}{1-\widehat{D_{d-1}}(k)}-\frac{d-1}{d+C}\, dk\nonumber \\
 & =\frac{d+C}{d-1}\cdot\frac{1}{(2\pi)^{d-1}}\int\frac{\widehat{D_{d-1}}(k)}{1-\widehat{D_{d-1}(k)}}\, dk+\frac{C+1}{d-1}.\label{eq:sumno0}\end{align}
To estimate the integral, apply Cauchy-Schwarz and get\[
\int\frac{\widehat{D}}{1-\widehat{D}}\le\left(\int\widehat{D}^{2}\right)^{1/2}\left(\int\frac{1}{(1-\widehat{D})^{2}}\right)^{1/2}.\]
The first integral (with the $(2\pi)^{1-d}$ which we have omitted from
the formula above) is simply the probability that simple random walk
returns to zero after two steps, and hence it is simply $1/2(d-1)$.
The second integral is shown in \cite[(3.4)-(3.6)]{HHS} to be bounded
independently of the dimension (\cite{HHS} have an additional parameter
in the calculation, $s$, which in our case is $1$). All in all we
get\[
\frac{1}{(2\pi)^{d-1}}\int\frac{\widehat{D_{d-1}}(k)}{1-\widehat{D_{d-1}}(k)}\,
dk\le
\left(\frac{1}{2(d-1)}\right)^{1/2}C^{1/2}
\]
which we plug into (\ref{eq:sumno0}) and get\[
\sum_{n\ne0}\mathbb{P}(\vec{0}\leftrightarrow(0\dotsc,0,n))\le\frac{d+C}{d-1}\cdot\frac{C}{\sqrt{2(d-1)}}+\frac{C+1}{d-1}\le\frac{C}{\sqrt{d}}\]
as required.\end{proof}

\section*{Proof of theorem \ref{thm:cutsets}}

Fix some vertex $v$ and examine the number $\Nn$ of infinite
components that intersect $\{v\}\times\ZZ$. $\Nn$ is invariant to
the translations of $\ZZ$ so it is constant almost surely. Further,
the standard modification argument \cite{BK89} shows that $\Nn$ cannot take any
finite value $>1$. The main step is to preclude the possibility that
$\Nn=\infty$. An infinite cluster which intersects $\{v\}\times\ZZ$
could intersect it either at finitely many vertices or in infinitely many
vertices. We start with the first case.

\begin{lem}\label{lem:nofinite}For every vertex $v$ of $G$, the probability that there
  exists an infinite cluster intersecting $\{v\}\times\ZZ$ at finitely
  many vertices only is 0.
\end{lem}
\begin{proof}
The idea is simple: if such clusters exist then they have some
positive density. However, as you trace the cluster from $v$ further and
further in $G$, its boundary must increase, eventually increasing
beyond $K/(\textrm{the density})$, leading to a contradiction to the
disconnection property of $G$.

Let us make this more formal. 
Using the disconnection property of $G$ repeatedly, one may find a sequence of
$v\in Q_1\subset Q_2\subset\dotsb$ with $|\partial Q_i|\le K$ and such
that every vertex $w$ which belong to some edge in $\partial Q_i$
belongs to $Q_{i+1}$. Denote by $\parv$ the internal vertex boundary
i.e.\ $\parv X$ is the set of all vertices in $X$ with neighbours outside
$X$. Then the events
\[
E_i:=\{(v,0)\leftrightarrow \parv Q_i\times\ZZ\}
\setminus\{(v,0)\leftrightarrow \parv Q_{i+1}\times\ZZ\}
\]
are disjoint, and hence $\sum\PP(E_i)\le 1$ and in particular $\PP(E_i)\to 0$. Fix now some $L$ and let
$F_i$ be the event that 
\[
0<|\{n\in \ZZ : \exists x\in\parv Q_i \textrm{ s.t. }
(v,0)\leftrightarrow (x,n) \textrm{ in }Q_i\times\ZZ\}|\le L
\]
(as usual, ``$a\leftrightarrow b$ in $X$'' means that there exists an
open path from $a$ to $b$ using only vertices in $X$). Then
if $F_i$ happened then there are at most $KL$ edges through which the
cluster may continue, and if they are all closed this would imply
$E_i$. Hence $\PP(E_i)\ge (1-p)^{LK}\PP(F_i)$. Hence
$\PP(F_i)\to 0$.

Assume now that the probability that the cluster of $(v,0)$  is
infinite, but its intersection with $\{v\}\times\ZZ$ is finite, is
positive. Let $r$ be some number such that 
\[
q:=\PP\Big(|\Cl(v,0)|=\infty, (\Cl(v,0)\cap\{v\}\times\ZZ)\subset
\{v\}\times[-r,r)\Big)>0.
\]
Fix $L=8Kr/q$ (recall that $K$ is the constant in the disconnection
property of $G$) and with this $L$ define the event $F_i$ above. Since
$\PP(F_i)\to 0$, let $i$ be sufficiently large such that
$\PP(F_i)<\frac14 q$. Subtracting we get
\[
\PP\left(\begin{array}{l}
|\Cl(v,0)|=\infty,\\
(\Cl(v,0)\cap\{v\}\times\ZZ)\subset\{v\}\times[-r,r),\\
|\Cl(v,0)\cap\parv Q_i\times\ZZ|> L
\end{array}\right)>\tfrac34 q
\]
(we used here that if $\Cl(0)$ is infinite then it cannot be contained
in $Q_i\times\ZZ$, except with probability 0, because
$p_c(\textrm{finite graph}\times\ZZ)=1$). Finally, strengthen the last requirement to $|\Cl(0)\cap \parv
Q_i\times[-N,N]|> L$ for some $N$ so large so that it only decreases
the probability by $\frac14 q$. We get
\[
\PP\left(\begin{array}{l}
|\Cl(v,0)|=\infty,\\
(\Cl(v,0)\cap\{v\}\times\ZZ)\subset\{v\}\times[-r,r),\\
|\Cl(v,0)\cap\parv Q_i\times[-N,N]|> L
\end{array}\right)>\tfrac12 q.
\]
Denote this event by $B$ and its translation by $n$ by $B_n$. We
finish by ergodicity of the translations by $2r\ZZ$. Indeed, we know
that 
\[
|\{n\in{1,\dotsc,a}:B_{2rn}\textrm{ occurred}\}|>\tfrac12 aq
\]
for $a$ sufficiently large (random). But this means that in $\parv
Q_i\times[2r-N, 2ar+N]$ there are $>\frac12 aq \cdot L=4Kar$ distinct points,
since each cluster has $> L$ points, and the clusters are disjoint
since their intersections with $\{v\}\times\ZZ$ belong to disjoint
intervals. But there is no room for $4Kar$ points, only to
$K(2ar-2r+1+2N)= 2Kar+o(a)$,
for $a$ sufficiently large, leading to a contradiction and
establishing the lemma.
\end{proof}
\begin{lem}It is not possible for infinitely-many clusters to
  intersect $\{v\}\times\ZZ$ at infinitely many vertices each.
\end{lem}
\begin{proof} The idea is as follows: we construct a 3-regular tree of
trifurcation points in $\{v\}\times\ZZ$ (as was done in
\cite{BLPS99}), and show that this contradicts the amenability of
$\ZZ$. Let us give the details.

We first show that there are trifurcation points. Following \cite{BK89} we note that if there are infinitely-many
clusters that intersect $\{v\}\times\ZZ$, then there are
trifurcation points. Indeed, for some $r\ge 1$
there is positive probability that three different infinite components
intersect $\{v\}\times [-r,r]$. Let $y_1,y_2,y_3$ be three different
points in $\{v\}\times [-r,r]$ which are connected to infinity by
three simple open paths which do not intersect $\{v\}\times [-r,r]$
again and are contained in three different components. Now modify the
environment as follows: open all edges of $\{v\}\times [-r,r]$ and
close all edges of $\{e\}\times[-r,r]$ for all edges $e\ni v$ except
the three edges which connect each $y_i$ to the next vertex in
$\gamma_i$. It is clear that after this modification one of the $y_i$
(the middle one) is a trifurcation point. We showed that there is positive probability that
there is a trifurcation point in $\{v\}\times [-r,r]$ and hence by
translation invariance the probability that $(v,n)$ is a trifurcation
point is positive for any $n$.

We remark that, from lemma \ref{lem:nofinite} we can deduce that for
any trifurcation point, each of the clusters one would get by removing
the trifurcation point must intersect $\{v\}\times\ZZ$ at
infinitely-many vertices.

We now follow \cite[\S 4]{BLPS99}, which shows that under these conditions
one may find a $3$-regular tree of trifurcation points. Since
\cite{BLPS99} explains the argument clearly, here we will brief. The
first step is 
\begin{clm*}If $(v,n)$ is a trifurcation point then each infinite
  cluster left after removal of $(v,n)$ has at least one other
  trifurcation point $(v,m)$.
\end{clm*}
\begin{proof} Define a mass transport function $M(n,m)$ for any $m$, $n\in\ZZ$ as
  follows: $M(m,n)=1$ if $(v,n)\leftrightarrow(v,m)$ and if $(v,n)$ is
  the unique closest trifurcation point to $(v,m)$ (the distance is the graph
  distance on the cluster). Let $M=0$ otherwise. Then
  each $m$ sends at most 1 unit of mass, and by the mass transport
  principle the expected amount of mass received by $n$ should also be
  no more than
  1 (we are using the mass transport principle on $\ZZ$ here). But the negation of the statement of the claim means that $n$
  receives an infinite amount of mass (here is where we use the remark
  above that each of the three clusters remaining after removal of
  $(v,n)$ is not just infinite, but intersects $\ZZ$ at
  infinitely-many vertices), so this must happen with probability 0. 
\end{proof}
As promised, we now construct an auxiliary
graph $T$ over the trifurcation points as follows: $m$ will be connected
to $n$ (denoted by $m\sim n$) if both are trifurcation points and if $(v,m)$ is the closest
trifurcation point to $(v,n)$ in one of the clusters remaining after
removal of $(v,n)$, or vice versa. Again ``closest'' means in the
graph distance, and we break ties by adding i.i.d.\ variables $X_n$
uniform in $[0,1]$ and choosing the one with the larger $X$. As in
\cite{BLPS99}, $T$ is a 3-regular tree.

Finally we derive a contradiction to the amenability of $\ZZ$. Denote $q=\PP(\textrm{0 is a}\linebreak[1]\textrm{trifurcation})$. Let $N$ be some
parameter and define the event $E(N)$ that $0$ is a trifurcation point
and in additional the three $v\sim 0$ satisfy 
$v\in[-N,N]$. Taking $N$ to be sufficiently large one may assume that
$\PP(E(N))>\frac34 q$. We fix $N$ at this value (and remove it from the
notation $E(N)$, so that we just denote it by $E$ from now). Denote
$E_n$ the translation of $E$ by $n$ i.e.\ the event that $n$ is a
trifurcation point with the condition on the neighbours as above.

Using ergodicity we know that as $r\to\infty$, the number of
trifurcation points in $[-r,r)$ is $2r(q+o(1))$, while the number of
$E_n$ for $n\in[-r,r)$ is $>2r\cdot \frac34 q$. Denote by $A$ the set of
these trifurcation points. Since $T$ is a 3-regular tree we know that
\[
|\{v\in T\setminus A\textrm{ such that }\exists w\in A, w\sim v\}|\ge
|A|+2.
\]
But this is clearly impossible, since any such $v$ must be in
$[-r-N,r+N]$, and there are $< 2N+2rq(\frac14 +o(1))$ trifurcation points
in this interval which are not in $A$. This is a contradiction and
the lemma is proved. With lemma \ref{lem:nofinite}, and since $v$ was
arbitrary, the theorem has been demonstrated.
\end{proof}

\subsection*{Acknowledgements}
We wish to thank Alain-Sol Sznitman for his question and Remco van der
Hofstad for help with the proof of lemma \ref{lem:Remco}. Both authors
partially supported by the Israel Science Foundation.

\end{document}